\documentclass{amsart} % please use amsart at 11pt
\usepackage{amssymb,latexsym}
\usepackage[OT2,T1]{fontenc}

\theoremstyle{plain}
\newtheorem{theorem}{Theorem}[section]
\newtheorem{lemma}[theorem]{Lemma}
\newtheorem{corollary}[theorem]{Corollary}
\newtheorem{proposition}[theorem]{Proposition}

\theoremstyle{definition}
\newtheorem{definition}{Definition}[section]

\theoremstyle{remark}
\newtheorem{remark}{Remark}[section]

\DeclareMathOperator{\ind}{ind}%
\DeclareMathOperator{\Proj}{Proj}%
\DeclareMathOperator{\ran}{ran}%
\DeclareMathOperator{\coker}{coker}%
\newcommand{\N}{\mathbf N}

\newcommand{\e}{\varepsilon}

\newcommand{\skp}[2]{\left<#1,#2\right>}

\renewcommand{\a}{\alpha}

\numberwithin{equation}{section} % to get equations numbered
\begin{document}
\title[Fredholm operators on $C^*$-algebras]{Fredholm operators on $C^*$-algebras} % please provide
                                % an abbreviated title

\author{Dragoljub J. Ke\v cki\' c}
\address{University of Belgrade\\ Faculty of Mathematics\\ Student\/ski trg 16-18\\ 11000 Beograd\\ Serbia}

\email{keckic@matf.bg.ac.rs}

\author{Zlatko Lazovi\'c}

\address{University of Belgrade\\ Faculty of Mathematics\\ Student\/ski trg 16-18\\ 11000 Beograd\\ Serbia}

\email{zlatkol@matf.bg.ac.rs}

\thanks{The authors was supported in part by the Ministry of education and science, Republic of Serbia, Grant \#174034.}

\begin{abstract}
The aim of this note is to generalize the notion of Fredholm operator to an arbitrary $C^*$-algebra. Namely,
we define "finite type" elements in an axiomatic way, and also we define Fredholm type element $a$  as such
element of a given $C^*$-algebra for which  there are finite type elements $p$ and  $q$ such that
$(1-q)a(1-p)$ is "invertible". We derive index theorem for such operators. In applications we show that
classical Fredholm operators on a Hilbert space, Fredholm operators in the sense of Breuer, Atiyah and Singer
on a properly infinite von Neumann algebra, and Fredholm operators on Hilbert $C^*$-modules over an unital
$C^*$-algebra in the sense of Mishchenko and Fomenko are special cases of our theory.
\end{abstract}

%\dedicatory{This paper is dedicated to Professor X on his 125th birthday.}



\maketitle

\section{Introduction}

Fredholm operators have been investigated for many years. Initially, they was considered as those operators
acting on a Banach space with finite dimensional kernel and cokernel. Their index is defined as a difference
of dimensions of the kernel and cokernel. The most important properties of Fredholm operators are the
following

1. The index theorem. It asserts that given two Fredholm operators $T$ and $S$, the operator $TS$ is also
Fredholm and
$$\ind(TS)=\ind T+\ind S$$

2. Theorem of Atkinson. It asserts that the operator $T$ is Fredholm if and only if it is invertible modulo
compact operators, or equivalently if and only if its image in the Calkin algebra $B(X)/C(X)$ is invertible,
where $C(X)$ denote the algebra of compact operators.

3. Perturbation theorem. If $T$ is a Fredholm operator, then the operator $T+K$ is Fredholm as well, provided
that $K$ is compact. In this case $\ind(T+K)=\ind T$.

4. Continuity of index. The Fredholm operators form an open set in the norm topology and the index is
constant on each connected component of this set.

The generalization of Fredholm theory to the level of von Neumann algebras was initially done by Breuer
\cite{Breuer68,Breuer69}, but it became famous after Atiyah's work \cite{Atiyah}. In order to generalize
earlier result of him and Singer, the well known index theorem, to noncompact manifolds he considered the
operators with kernel and cokernel that don't belong to finite dimensional subspaces, but affiliated to some
von Neumann algebra. He defined the dimension of such subspaces as the trace of the corresponding projection
in the appropriate von Neumann algebra, see \cite{Atiyah}. Soon after that Mischenko and Fomenko introduced
the notion of Fredholm operator in the framework of Hilbert $C^*$ modules, see \cite{Mischenko}. It was
repeated with a different approach by Mingo, see \cite{Mingo} and also important references therein. Several
decades ago, there are, also, some attempts to establish axiomatic Fredholm theory in the framework of von
Neumann algebras $G$ (relative to ideal $\mathcal I$). Fredholm operators are defined, in this case, as
invertible elements in the quotient space $G/\mathcal {I}$, and the index is defined as an element from
$C(\Omega)$ (the set of all continuous function on $\Omega$, where $C(\Omega)$ is isomorphic to the center
$Z$ of $G$), see \cite{Coburn}, \cite{Olsen} and references therein. There are, also, other attempts in
generalization of Fredholm theory, for instance \cite{Alvarez}.

The aim of this note is to single out the properties of finite rank operators which ensure the development of
Fredholm theory. In other words we shall introduce axiomatic foundation of Fredholm theory. We shall deal
within a framework of a unital $C^*$-algebra $\mathcal A$ and its faithful representation $\rho:\mathcal A\to
B(H)$ as the subalgebra $\rho(\mathcal A)$ of the algebra of all bounded operators on some Hilbert space $H$.
Further, we prove that the standard Fredholm operators, Fredholm operators in the sense of Atiyah and Singer
on $II_\infty$ factors, and Fredholm operators in the sense of Mishchenko and Fomenko  are special cases of our theory.

Throughout this paper we shall assume that a given $C^*$-algebra is always unital, even if this is not
explicitly mentioned. By $\ker$ ($\coker$) we shall denote the kernel(cokernel) of some operator. If $X$ is a
closed subspace of some Hilbert space $H$, $P_X$ will denote the orthogonal projection on $X$. Sometimes we
shall omit the word {\em orthogonal}, i.e, projection will mean orthogonal projection. The letter $1$ will
stand for unit element in an abstract algebra, whereas $I$ will stand for identity operator on some Hilbert
space. Similarly, we use small letters $a$, $b$, $t$, etc.\ to denote elements of an abstract algebra, whereas we use
capital letters $A$, $B$, $T$, etc.\ to denote operators on a concrete Hilbert space.

\begin{definition}\label{definicija konacnih}Let $\mathcal A$ be an unital $C^*$-algebra, and let $\mathcal F\subseteq\mathcal A$ be a
subalgebra which satisfies the following conditions:

$(i)$ $\mathcal F$ is a selfadjoint ideal in $\mathcal A$, i.e.~for all $a\in\mathcal A$, $b\in\mathcal F$
there holds $ab,\,ba\in\mathcal F$ and $a\in\mathcal F$ implies $a^*\in\mathcal F$;

$(ii)$ There is an approximate unit $p_{\alpha}$ for $\mathcal F$, consisting of projections;

$(iii)$ If $p,q\in\mathcal F$ are projections, then there exists $v\in\mathcal A$, such that $vv^*=q$ and
$v^*v\bot p$, i.e.~$v^*v+p$ is a projection as well;

Such a family we shall call {\em finite type elements}. In further, we shall denote it by~${\mathcal F}$.
\end{definition}

\begin{definition}\label{Definicija K grupe}
Let $\mathcal A$ be an unital $C^*$-algebra, and let $\mathcal F\subseteq\mathcal A$ be an algebra of finite type elements.

In the set $\Proj(\mathcal F)$ we define the equivalence relation:
$$p\sim q\quad\Leftrightarrow\quad\exists v\in\mathcal A\:\:vv^*=p,\:\:v^*v=q,$$
i.e.~Murray - von Neumann equivalence. The set $S(\mathcal F)=\Proj(\mathcal F)/\sim$ is a commutative
semigroup with respect to addition, and the set $K(\mathcal F)=G(S(\mathcal F))$, where $G$ denotes the
Grothendic functor, is a commutative group.

\end{definition}

\section{Results}

We devide this section into four subsections. First three of them, {\em Well known Hilbert space Lemmata},
{\em Almost invertibility} and {\em Approximate units technique} contain introductory material necessary for
the last one, {\em Index and its properties}, which contains the main results.

\subsection*{Well known Hilbert space Lemmata}

First, we list, and also prove, three elementary statements concerning operators on a Hilbert space.

The first Lemma establishes that two projections, close enough to each other, are unitarily equivalent, and
also that obtained unitary is close to the identity operator.

\begin{lemma}\label{p-q manje 1} Let $P$, $Q\in B(H)$, let $H$ be a Hilbert space and let $||P-Q||<1$. Then there
are partial isometries $V$, $W$ such that
\begin{equation}\label{Ris-Nadj}
VV^*=P,\qquad V^*V=Q,
\end{equation}
$$WW^*=I-P,\qquad W^*W=I-Q.$$
Moreover, there is a unitary $U$ such that
$$U^*PU=Q,\qquad U^*(I-P)U=I-Q.$$

In addition, $V$, $W$, $U\in C^*(I,P,Q)$ - a unital $C^*$-algebra generated by $P$ and $Q$.

Finally, if $||P-Q||<1/2$, we have the estimate
\begin{equation}\label{ocena_unitarnog}
||I-U||\le C||P-Q||,
\end{equation}
where $C$ is an absolute constant, i.e.\ it does not depend neither on $P$ nor on $Q$.
\end{lemma}

\begin{proof} The existence of $V\in B(H)$ with properties (\ref{Ris-Nadj}) was proved in \cite[\S
105, page 268]{RisNadjbook}. Also, it is given that $V=P(I+P(Q-P)P)^{-1/2}Q\in C^*(I,P,Q)$.

Since $||(I-P)-(I-Q)||=||P-Q||$, we can apply previous reasoning to $I-P$ and $I-Q$ to obtain $W$ with
required properties. Clearly, $U=V+W$ is a unitary we are looking for.

Finally, to obtain (\ref{ocena_unitarnog}) we have
$$||Q-V||=||Q-P(I+P(Q-P)P)^{-1/2}Q||\le||Q-P(I+P(Q-P)P)^{-1/2}||.$$
Let $\a_n$ be the coefficients in Taylor expansion of the function $t\mapsto(1+t)^{-1/2}$. Then
$(I+P(Q-P)P)^{-1/2}=I+\sum_{n=1}^\infty\a_n(P(Q-P)P)^n$ and hence
\begin{multline*}||Q-V||\le||Q-P-\sum_{n=1}^\infty\a_n(P(Q-P)P)^n||\le||Q-P||+\sum_{n=1}^\infty|\a_n|\,||Q-P||^n=\\
    ||Q-P||(1+\sum_{n=1}^\infty|\a_n|\,||Q-P||^{n-1})\le C_1||Q-P||,
\end{multline*}
where $C_1=1+\sum_{n=1}^\infty|\a_n|/2^{n-1}<+\infty$. Thus
$||I-U||=||I-Q+Q-W-V||\le||Q-V||+||I-Q-W||\le2C_1||Q-P||$. (Indeed, since $C_1$ does not depend on $P$, $Q$,
applying the previous reasoning to the projections $I-P$, $I-Q$ we have $||I-Q-W||\le C_1||(I-Q)-(I-P)||$.)
\end{proof}

The second Lemma characterizes "almost orthogonal" projections.

\begin{lemma}\label{norma od PQ}
Let $P$, $Q\in B(H)$ be projections, and let for all $\xi\in P(H)$ and all $\eta\in Q(H)$  holds
$$|\skp\xi\eta|\le c||\xi||\,||\eta||.$$
Then $||PQ||\le c$.
\end{lemma}

\begin{proof} Let $\xi$, $\eta\in H$. Then $P\xi\in P(H)$, $Q\eta\in Q(H)$, and
$$|\skp{QP\xi}\eta|=|\skp{P\xi}{Q\eta}|\le c||P\xi||\,||Q\eta||\le c||\xi||\,||\eta||.$$
Therefore $||QP||\le c$ and also $||PQ||=||(QP)^*||\le c$.
\end{proof}

The last Lemma in this subsection deals with the adjoint of an isomorphism between two different subspaces.

\begin{lemma}\label{adjointy is bounded below}
Let $H_1$, $H_2$ be Hilbert spaces, $K\le H_1$, $L\le H_2$. Also, let $T:H_1\to H_2$ be a bounded mapping
that maps $K$ bijectively to $L$ and $T|_{K^\perp}=0$. Then $T^*$ maps bijectively $L$ to $K$, and
$T^*|_{L^\perp}=0$.
\end{lemma}

\begin{proof} Since $\ker T^*=(\ran T)^\perp=L^\perp$, it follows $T^*|_{L^\perp}=0$. In particular, $T^*$ is
injective on $L$. Also, $\overline{\ran T^*}=(\ker T)^\perp=K$, i.e. $\ran T^*$ is dense in $K$.

Since $T$ maps bijectively $K$ to $L$ ($K$, $L$ closed), by open mapping theorem, $T$ is bounded below on
$K$, i.e.\ there is $c>0$ such that $||T\xi||\ge c||\xi||$ for all $\xi\in K$. Let $\eta\in L$. Then
$T^*\eta\in K$ and
\begin{multline*}||T^*\eta||=\sup_{\substack{\xi\in K\\||\xi||=1}}|\skp{\xi}{T^*\eta}|=
    \sup_{\substack{\xi\in K\\||\xi||=1}}|\skp{T\xi}{\eta}|=
    \sup_{\substack{\xi\in K\\||\xi||=1}}||T\xi||\,\left|\skp{\frac{T\xi}{||T\xi||}}{\eta}\right|\ge\\
    c\sup_{\substack{\xi\in K\\||\xi||=1}}\left|\skp{\frac{T\xi}{||T\xi||}}{\eta}\right|=
    c\sup_{\substack{\eta'\in L\\||\eta'||=1}}\left|\skp{\eta'}{\eta}\right|=c||\eta||.
\end{multline*}
The last equality is due to $T(K)=L$.

Therefore, $T^*$ is topologically injective on $L$. Hence its range is closed and we are done.
\end{proof}

\subsection*{Almost invertibility}

In this subsection, we introduce the notion of almost invertibility, i.e.\ invertibility up to a pair of
projections, following definitions of Fredholm operators in the framework of Hilbert $C^*$-modules
\cite[Definition 2.7.4]{TrMa}, and \cite[Chapter 17]{Wegge}.

\begin{definition}
Let $a\in\mathcal A$ and let $p$, $q\in\mathcal F$. We say that $a$ is invertible up to pair $(p,q)$ if the
element $a'=(1-q)a(1-p)$ is invertible, i.e., if there is some $b\in \mathcal A$ with $b=(1-p)b(1-q)$ (and
immediately $bq=0$, $pb=0$, $b=(1-p)b=b(1-q)$) such that
$$a'b=1-q,\qquad ba'=1-p.$$
We refer to such $b$ as almost inverse of $a$, or $(p,q)$-inverse of $a$.
\end{definition}

First, we establish that almost invertible elements make an open set, and that almost inverse of $a$ (in a
certain sense) continuously changes with respect to $a$, as well as with respect to $p$ and $q$

\begin{lemma}\label{mala pretumbacija} Let $a$ be invertible up to $(p,q)$ and let $b$ be $(p,q)$-inverse of $a$.

a) The element $a+c$ is also invertible up to $(p,q)$ for every $c\in A$, which satisfies $||c||<||b||^{-1}$.
If $b_1$ denote $(p,q)$-inverse of $a+c$, then
\begin{equation}\label{ocena_inverza_perturbacije}
||b_1||\le\frac{||b||}{1-||b||\,||c||}.
\end{equation}

b) If $||p-p'||$, $||q-q'||<\min\{1/2,1/4C||a||\,||b||\}$, where $C$ is the constant from
(\ref{ocena_unitarnog}) then $a$ is also invertible up to $(p',q')$. Moreover, if $b'$ is $(p',q')$-inverse
of $a$, then $||b'||\le 2||b||$.
\end{lemma}

\begin{proof} a) Let $b$ be $(p,q)$-inverse of $a$. Then  $(1-q)(a+c)(1-p)=a'+c'=a'(1-p)+(1-q)c=a'(1-p)+a'bc'=a'(1-p+bc')$, where
$c'=(1-q)c(1-p).$ The element $bc'$ belongs to $(1-p)\mathcal A(1-p)$ and its norm is
$||bc'||\le||b||\,||c||<1$, for $||c||<||b||^{-1}$. Therefore, $1-p+bc'$ is invertible in the corner algebra
$(1-p)\mathcal A(1-p)$. Denote its inverse by $t$. We have $(1-q)(a+c)(1-p)tb=a'(1-p+bc')tb=a'b=1-q$ and
$tb(1-q)(a+c)(1-p)=tba'(1-p+bc')=t(1-p)(1-p+bc')=1-p$. Therefore $b_1=tb$ is $(p,q)$-inverse of $a+c$, for
$||c||<||b||^{-1}$.

Let us prove (\ref{ocena_inverza_perturbacije}). We have
$$||t||=||(1-p)+\sum_{n=1}^\infty(-1)^n(bc')^n||\le1+\sum_{n=1}^\infty||b||^n||c'||^n=\frac1{1-||b||\,||c'||}.$$
Since $||c'||\le||c||$ and $b_1=tb$, (\ref{ocena_inverza_perturbacije}) follows.

b) By Lemma \ref{p-q manje 1} there is a unitary $u$ such that $u^*q'u=q$. Then, we have
$$u^*(1-q')a(1-p)=(1-q)a(1-p)-u^*(1-q')(u-1)a(1-p)=(1-q)a(1-p)-c.$$
Note that $c\in(1-q)\mathcal A(1-p)$, as well as $||c||<||u-1||\,||a||<C||q-q'||\,||a||<1/(4||b||)$. By the
previous part a), we have that there is $(p,q)$-inverse of $u^*(1-q')a(1-p)$, say $b_1$. It is easy to check
that $b_1u^*$ is $(p,q')$-inverse of $a$, and also
$$||b_1u^*||\le||b_1||<\frac{||b||}{1-||b||\,||c||}\le\frac{4||b||}3.$$

In a similar way we can substitute $p$ with $p'$ to obtain that $a$ is invertible up to $(p',q')$. Denote its
$(p,q)$-inverse by $b'$. Also we have the inequality
$$||b'||\le\frac{||b_1||}{1-||b_1||\,||c'||}\le\frac{4||b||/3}{1-4||b||\,||c'||/3}\le 2||b||.$$
\end{proof}

The next Lemma will be used in the sequel many times. It allows us to transfer a projection from the right
side of an element to its left side, if the considered element is invertible on this projection. In the
framework of a representation, it means that we can carry a projection from the domain of an operator to its
codomain. The transferred projection is equivalent to the initial one.

\begin{lemma}\label{lema 2.4.}
Let $\mathcal A$ be a unital $C^*$-algebra, let $a\in\mathcal A$, let $p$, $r\in\mathcal A$ be projections
such that $a$ is left invertible up to $p$, i.e.\ there is $b\in\mathcal A$ such that $ba(1-p)=1-p$. Also,
let $r\le 1-p$. Then:

$(i)$ $ar$ has a polar decomposition in $\mathcal A$, i.e.
\begin{equation}\label{polarna}
ar=v|ar|,\quad v,|ar|\in\mathcal A,\quad r=vv^*.
\end{equation}
Further, $s=v^*v\in\mathcal A$ is the minimal projection in $\mathcal A$, such that $sar=ar$. (Obviously,
$s\sim r$, and if $r\in\mathcal F$, then $s\in\mathcal F$ as well.) Moreover, $a$ is invertible up to
$(1-r,1-s)$.

$(ii)$ If $\mathcal A\overset\rho\hookrightarrow B(H)$ is any faithful representation of $\mathcal A$, then $L=\rho(ar)(H)$ is a closed
subspace, and $\rho(s)$ is the projection on $L$. In particular projection on $L=\rho(ar)(H)$ belongs to
$\rho(\mathcal A)$.
\end{lemma}

\begin{proof}
Let $\mathcal A\overset\rho\hookrightarrow B(H)$ be some faithful representation of $C^*$-algebra $\mathcal
A$ and let $\rho(r)(H)=K$. Denote $A=\rho(a)$, $B=\rho(b)$, $P=\rho(p)$, $R=\rho(r)$. Obviously, $AR$ has a
polar decomposition $AR=V|AR|$ in $B(H)$. Since the partial isometry $v$ is uniquely determined by
(\ref{polarna}), we have only to prove that $V\in\rho(\mathcal A)$.

Let $H_1=(I-P)(H)$. For $\xi\in H_1$ we have
$$||\xi||=||BA(I-P)\xi||\le||B||\,||A(I-P)\xi||=k^{-1}||A\xi||,$$
where $k=||B||^{-1}$. (We used $(I-P)\xi=\xi$.) Thus,
$$||A\xi||\ge k||\xi||,\qquad\xi\in H_1,$$
i.e.\ $A$ is injective and bounded below on $H_1$, and consequently on $K=R(H)\le H_1$ ($r\le 1-p$).

Therefore, $L=AR(H)=A(K)$ is a closed subspace. (This proves the first claim in $(ii)$.) As $AR$ maps
bijectively $K$ to $L$, $(AR)^*$ maps bijectively $L$ to $K$ - Lemma \ref{adjointy is bounded below}, and
hence, $RA^*AR=(AR)^*AR$ is an isomorphism of $K$. Denote its inverse by $\hat T:K\to K$. Let $T:H\to H$ be
the extension of $\hat T$, defined to be $0$ on $K^\perp$. Trivially, $T\ge0$, and $\sqrt T$ exists. Also,
$\sqrt T$ maps $K$ to $K$ and $\sqrt T|_{K^\perp}=0$. We claim that $V=AR\sqrt T$. Indeed, both $V$ and
$AR\sqrt T$ annihilates $K^\perp$, whereas on $K$, $\sqrt T$ and $|AR|$ are inverses to each other, since
$|AR|=((AR)^*AR)^{1/2}$ and $T$ is inverse for $(AR)^*AR$. Thus, it remains to prove that $T\in\rho(\mathcal
A)$.

As it is easy to check, $(AR)^*AR+I-R$ and $T+I-R$ are inverses to each other. Since $(AR)^*AR+I-R\in\rho(\mathcal
A)$, its inverse also belongs to $\rho(\mathcal A)$. Namely, by Theorem 11.29 from \cite{Rudin} (or Theorem
2.1.11 from \cite{Marfi}) the spectrum of an element is the same with respect to algebra as well as with
respect to subalgebra. Apply this conclusion to $0$ as the point of spectrum to obtain $T+I-R\in\rho(\mathcal
A)$ and immediately, $T\in\rho(\mathcal A)$. Therefore $T=\rho(t)$ for some $t\in\mathcal A$. Hence
$V=\rho(v)$, where $v=ar\sqrt t$ and $ar$ has a polar decomposition in $\mathcal A$.

To finish the proof, note that $sar=ar$, since $s=v^*v$ and $S=\rho(s)$ is the projection on the subspace
$L=AR(H)$. If for some other projection $s_1$, there holds $s_1ar=ar$ then must be $S_1(H)\ge L$ and hence
$S_1\ge S$ implying $s_1\ge s$.
\end{proof}

The following Lemma (as many others), in fact, deals with $2\times2$ matrices. It can be freely reformulated
as follows: If $[a_{i,j}]_{i,j=1}^2$ is invertible, $a_{11}$ is invertible and if this matrix is either
triangular ($a_{12}=0$) or close to triangular ($a_{12}$ small) then $a_{22}$ is also invertible.

However, we do not use matrix notation (now and in the sequel) in order to avoid confusions such as with
respect to which pair(s) of projections a matrix is formed.

\begin{lemma}\label{razlaganje} Let $a$, $p$, $q\in\mathcal A$, let $p$, $q$ be projections such that $a$ is invertible up to
$(p,q)$ and let $r_1$, $r_2$, $s_1$, $s_2$ be projections such that $1-p=r_1+r_2$, $1-q=s_1+s_2$. If $a$ is
invertible up to $(1-r_1,1-s_1)$ and if $||s_2ar_1||<||b||^{-1}$ or $||s_1ar_2||<||b||^{-1}$, where $b$ is
$(p,q)$ inverse of $a$, then $a$ is invertible up to $(1-r_2,1-s_2)$.
\end{lemma}

\begin{proof} $1^\circ$ case - let $s_2ar_1=0$. Decompose $(1-q)a(1-p)$ and $b$ as
\begin{equation}\label{decompose a}
(1-q)a(1-p)=(s_1+s_2)a(r_1+r_2)=s_1ar_1+s_1ar_2+s_2ar_2.
\end{equation}
\begin{equation}\label{decompose b}
b=(1-p)b(1-q)=(r_1+r_2)b(s_1+s_2)=r_1bs_1+r_1bs_2+r_2bs_1+r_2bs_2.
\end{equation}

Multiplying (\ref{decompose a}) by (\ref{decompose b}) from right we get
$$s_1+s_2=1-q=s_1ar_1bs_1+s_1ar_1bs_2+s_1ar_2bs_1+s_1ar_2bs_2+s_2ar_2bs_1+s_2ar_2bs_2.$$
Multiplying the last equality by $s_2$ from both, left and right side we obtain:
\begin{equation}\label{desni inverz}
s_2=s_2ar_2r_2bs_2.
\end{equation}

Multiplying (\ref{decompose a}) and (\ref{decompose b}) in reverse order we get
\begin{equation}\label{pomnozeno naopako}
r_1+r_2=1-p=r_1bs_1ar_1+r_1bs_1ar_2+r_1bs_2ar_2+r_2bs_1ar_1+r_2bs_1ar_2+r_2bs_2ar_2.
\end{equation}
Multiply (\ref{pomnozeno naopako}) by $r_2$ from the left, and by $r_1$ from the right to obtain
$r_2bs_1ar_1=0$. The element $a$ has an $(1-r_1,1-s_1)$ inverse, say $b'$. It holds $s_1ar_1b'=s_1$, and we
have $r_2bs_1=r_2bs_1ar_1b'=0b'=0$.

Finally, multiply (\ref{pomnozeno naopako}) by $r_2$ from both, left and right side to obtain
$$r_2=r_2bs_1ar_2+r_2bs_2ar_2=0ar_2+r_2bs_2ar_2=r_2bs_2s_2ar_2.$$
This together with (\ref{desni inverz}) means that $r_2bs_2$ is $(1-r_2,1-s_2)$-inverse for $a$. Therefore,
by Lemma \ref{lema 2.4.} applying to  $r_2\leq r_1+r_2$ we found that $r_2\sim s_2$.

$2^\circ$ case - general. Consider $\hat a=a-s_2ar_1$. Obviously, $s_1\hat a r_1=s_1ar_1$, $s_2\hat a r_2=s_2
ar_2$ and $s_2\hat ar_1=0$. Since $||s_2ar_1||<||b||^{-1}$, by Lemma \ref{mala pretumbacija} a), $\hat a$ is
invertible up to $(p,q)$ and we can apply the previous case.

If $||s_1ar_2||<||b||^{-1}$, apply the presented proof to $a^*$.
\end{proof}

The next Lemma is a refinement of Lemma \ref{lema 2.4.}.

\begin{lemma}\label{treci deo leme 2.4}
Let $\mathcal A$ be a unital $C^*$-algebra, let $p$, $q$, $r\in\mathcal A$ be projections such that $a$ is
invertible up to $(p,q)$, and $r\le 1-p$, and let $s$ be the projection obtained in Lemma \ref{lema 2.4.}. If
$qa(1-p)=0$, then $s\le1-q$, and $1-q-s\sim1-p-r$.
\end{lemma}

\begin{proof}
Since $a$ is invertible up to $(p,q)$, $a$ is left invertible up to $p$. So, the projection $s$ is well
defined.

If $qa(1-p)=0$, then $(1-q)a(1-p)=a(1-p)$ which implies $ar=a(1-p)r=(1-q)a(1-p)r=(1-q)ar$. So, $s\le1-q$ by
minimality of $s$.

Denote $r_1=1-p-r$. Apply the previous part of the proof to $r_1\le1-p$ to obtain the minimal $s_1\le1-q$
such that $s_1ar_1=ar_1$ and $r_1\sim s_1$. Denote $s_2=1-q-s$. Projections $s_1$ and $s_2$ might not
coincide, but they must be equivalent. More precisely $s_2\sim r_1\sim s_1$. To show this it is enough to
prove that $s_2ar_1$ is "invertible". However, from $sar=ar$ we get
$s_2ar=(1-q)ar-sar=(1-q)a(1-p)r-sar=a(1-p)r-qa(1-p)r-sar=ar-0-sar=0$. Thus, it is sufficient to apply Lemma
\ref{razlaganje}.
\end{proof}

The last statement in this subsection ensures that almost invertible elements can be triangularized. More
precisely, if $a$ is invertible up to $(p,q)$ then we can substitute $p$ or $q$ (one of them - not
simultaneously) by an equivalent projection such that almost invertibility is not harmed and such that with
respect to new projections the considered element has a triangular form.

\begin{proposition}\label{trijagularizacija}
Let $\mathcal A$ be a unital $C^*$-algebra, let $p$, $q\in\mathcal A$ and let $a$ be invertible up to
$(p,q)$.

a) Then there is a projection $q'\in\mathcal A$, $q'\sim q$ such that $a$ is invertible up to $(p,q')$ and
$q'a(1-p)=0$;

b) Also, there is another projection $p'\in\mathcal A$, $p'\sim p$ such that $a$ is invertible up to $(p',q)$
and $(1-q)ap'=0$.
\end{proposition}

\begin{proof} a) Let $b$ be the $(p,q)$-inverse of $a$. Then
\begin{equation}\label{svojstva b-a}
\begin{gathered}
(1-q)a(1-p)b=1-q,\quad b(1-q)a(1-p)=1-p,\\(1-p)b(1-q)=b,\quad pb=bq=0,\quad b=(1-p)b,\quad b=b(1-q).
\end{gathered}
\end{equation}
Let $u=1+qab$. This element has its inverse, $u^{-1}=1-qab$, which can be easily checked.

Consider the element $a_1=u^{-1}a=(1-qab)a=a-qaba$. Using $b(1-q)=b$ and $b(1-q)a(1-p)=1-p$ (see
(\ref{svojstva b-a})), we have
\begin{equation}\label{qa1(1-p)}
qa_1(1-p)=q(a-qaba)(1-p)=qa(1-p)-qab(1-q)a(1-p)=qa(1-p)-qa(1-p)=0.
\end{equation}

The element $u^{-1*}$ is invertible together with $u$ and we can apply Lemma \ref{lema 2.4.} to obtain a
minimal projection $q'\in\mathcal A$ such that $q'\sim q$, $u^{-1*}$ is invertible up to $(1-q,1-q')$ and
$q'u^{-1*}q=u^{-1*}q$. The last equality implies
\begin{equation}\label{definicijaq'}
qu^{-1}=qu^{-1}q',\qquad qu^{-1}(1-q')=0,
\end{equation}
which ensures, by Lemma \ref{razlaganje} that $u^{-1}$ is invertible up to $(1-q',1-q)$, as well as up to
$(q',q)$.

Let us, first prove that $q'a(1-p)=0$. Indeed, by (\ref{qa1(1-p)}) and (\ref{definicijaq'}) we have
\begin{equation}\label{jedn1}
0=qa_1(1-p)=qu^{-1}a(1-p)=qu^{-1}q'a(1-p).
\end{equation}
Let $t$ is $(1-q',1-q)$-inverse for $u^{-1}$. We have $tqu^{-1}q'=q'$.
Multiply the  equation \ref{jedn1} from the left by $t$ to obtain
$0=tqu^{-1}q'a(1-p)=q'a(1-p)$.

Finally, let us prove that $a$ is invertible up to $(p,q')$. Note that
$(1-q)a_1(1-p)=(1-q)(a-qaba)(1-p)=(1-q)a(1-p)$. Next, by (\ref{definicijaq'}) we get $qu^{-1}(1-q')=0$ and
$(1-q)u^{-1}(1-q')=u^{-1}(1-q')$ and hence
\begin{eqnarray*}
(1-q')a(1-p)&=&a(1-p)=uu^{-1}a(1-p)=ua_1(1-p)=u(1-q)a_1(1-p)\\
&=&u(1-q)a(1-p).
\end{eqnarray*}
Now, it is easy to check that $bu^{-1}(1-q')$ is $(p,q')$-inverse of $a$. Indeed
\begin{multline*}(1-q')a(1-p)\cdot bu^{-1}(1-q')=u(1-q)a(1-p)bu^{-1}(1-q')=\\
    =u(1-q)u^{-1}(1-q')=uu^{-1}(1-q')=1-q'
\end{multline*}
and
$$bu^{-1}(1-q')\cdot(1-q')a(1-p)=bu^{-1}u(1-q)a(1-p)=1-p.$$

b) It is enough to apply the previous conclusion to $a^*$.
\end{proof}

\subsection*{Approximate units technique}

For the development of the abstract Fredholm theory, we required in Definition \ref{definicija konacnih} an
ideal with an approximate unit consisting of projections. In this subsection we derive some properties of
such an ideal.

In the next two Lemmata we obtain that an approximate unit absorbs any other projections and that any finite
projection is an element of some approximate unit.

\begin{lemma}\label{ekvivalentan nekoj jedinici}Let $p_\a\in\mathcal F$ be an approximate unit, and let
$p\in\mathcal F$. Then there is some $\a_0$ and $p'$ such that $p\sim p'\le p_{\a_0}$. Moreover, $\a_0$ can
be chosen such that $||p-p'||$ is arbitrarily small.
\end{lemma}

\begin{proof} Since $p_\a$ is an approximate unit, we have $||p-p_\a p||\to0$, as $\a\to\infty$.
Therefore, there is $\a_0$ such that
\begin{equation}\label{manje od polovine}
||p-p_{\a_0}p||\le\delta<1,\qquad\mbox{implying}\qquad||p-pp_{\a_0}p||\le\delta.
\end{equation}

First, we obtain that $pp_{\a_0}p$ is invertible element of the corner algebra $p\mathcal A p$.

Now, pick a faithful representation $\rho$ of $\mathcal A$ on some Hilbert space $H$. Denote the images of
$\rho$ by the corresponding capital letters, $P=\rho(p)$, $P_{\a_0}=\rho(p_{\a_0})$ etc.

Let $K=P(H)$. Since $p_{\a_0}$ is invertible up to $(1-p,1-p)$ we can apply Lemma \ref{lema 2.4.} $(i)$ to
conclude that there is $p'\in\mathcal A$, $p'\sim p$. By the part $(ii)$ of the same Lemma, we get $p'\le
p_{\a_0}$ ($P'$ is the projection on the range of $P_{\a_0}P$. The last is, obviously, subspace of the range
of $P_{\a_0}$.)

It remains to prove that $||p-p'||$ can be arbitrarily small. To do this, let us prove that $p_{\a_0}$ does
not change norm of $\xi\in K$ too much. Indeed, for $\xi\in K$, using (\ref{manje od polovine}) we have
$$||p_{\a_0}\xi||=||p_{\a_0}p\xi||=||p\xi-(p-p_{\a_0}p)\xi||\le(1+\delta)||\xi||,$$
and also,
\begin{equation}\label{vece od 1-delta}
||p_{\a_0}\xi||=||p_{\a_0}p\xi||=||p\xi-(p-p_{\a_0}p)\xi||\ge||\xi||-||(p-p_{\a_0}p)\xi||\ge(1-\delta)||\xi||.
\end{equation}

Now, we want to prove that $1-p'$ and $p$ (and similarly $1-p$ and $p'$) are "almost orthogonal".

Let $\eta=(1-p)\eta$ and $\zeta=p'\zeta$. Then $\zeta=p_{\a_0}\xi$ for some $\xi=p\xi\in K$, and
$||\zeta||\ge(1-\delta)||\xi||$ by (\ref{vece od 1-delta}). We have
$$\skp\zeta\eta=\skp{p_{\a_0}p\xi}{(1-p)\eta}=-\skp{(p-p_{\a_0}p)\xi}{(1-p)\eta},$$
and hence
$$|\skp{\zeta}\eta|\le||p-p_{\a_0}p||\,||\xi||\,||\eta||\le\frac\delta{1-\delta}||\zeta||,||\eta||.$$
Therefore, by Lemma \ref{norma od PQ}
\begin{equation}\label{ocena p'-p'p}
||p'(1-p)||<\frac\delta{1-\delta}.
\end{equation}

Now, let $\zeta=(1-p')\zeta$ and $\eta=p\eta$. Then, $\zeta\perp p_{\a_0}p\eta$ (since $p_{\a_0}p\eta\in
p'H$) and therefore
\begin{multline*}
|\skp\zeta\eta|=|\skp{(1-p')\zeta}{p\eta}|=|\skp{(1-p')\zeta}{(p-p_{\a_0}p)\eta}|\le\\
    \le||p-p_{\a_0}p||\,||\zeta||\,||\eta||\le\delta||\zeta||\,||\eta||,
\end{multline*}
from which we conclude again by Lemma \ref{norma od PQ}
\begin{equation}\label{ocena p-pp'}
||p-p'p||=||(p-p'p)^*||=||p-pp'||\le\delta.
\end{equation}

From (\ref{ocena p'-p'p}) and (\ref{ocena p-pp'}) we get
$$||p-p'||\le||p-pp'||+||p'-pp'||\le\delta\frac{2-\delta}{1-\delta},$$
which can be arbitrarily small.
\end{proof}

\begin{lemma}\label{utapanje u jedinicu} Let $\mathcal F$  be an algebra of finite type elements.
For every projection $p\in\mathcal F$ there is an approximate unit $p_\a$ in $\mathcal F$ such that for
all $\a$ there holds $p\le p_\a$.
\end{lemma}

\begin{proof}
Let $p\in\mathcal F$ and let $p_\a$ be an approximate unit. By Lemma \ref{ekvivalentan nekoj jedinici}, for
$\a$ large enough, we have $p\sim p'\le p_\a$ and $||p-p'||<1$. By Lemma \ref{p-q manje 1} there is a unitary
$u$ such that $p'=u^*pu$. Then $p'_\a=u^*p_\a u$ is an approximate unit that contains $p$.
\end{proof}

The following Proposition plays the key role in this paper. It ensures that we can transfer an approximate
unit from the right side of an almost invertible element to its left side, retaining some triangular
properties. Briefly, if such almost invertible element is upper triangular, with upper left entry invertible,
then this entry itself has a triangular form (but lower - not upper) with respect to an approximate unit from
the right and its corresponding approximate unit from the left.

\begin{proposition}\label{desna apr jed}
Let $p$, $q\in\mathcal F$, let $a$ be invertible up to $(p,q)$, and let $qa(1-p)=0$. Further, let $p_\a\ge p$
be an approximate unit for $\mathcal F$. Then there exists an approximate unit $q_\a\in\mathcal F$, such that
$q_\a-q\sim p_\a-p$, $a$ is invertible up to $(p_\a,q_\a)$, and
\begin{equation}\label{prelaz sa p_a na q_a v1}
(1-q_\a)a(p_\a-p)=0,
\end{equation}
for all $\a$.
\end{proposition}

\begin{proof}a) Since $a'=(1-q)a(1-p)$ is invertible, and since $p_\a-p\le 1-p$, there is (by Lemma \ref{lema
2.4.}) a minimal projection $q'_\a$ such that $q'_\a a(p_\a-p)=a(p_\a-p)$ and $a$ is invertible up to
$(1-(p_\a-p),1-q'_\a)$. By $qa(1-p)=0$ and by Lemma \ref{treci deo leme 2.4} we have $q'_\a\le 1-q$. Set
$q_\a=q+q_\a'$. We have
\begin{equation}\label{prelaz sa p_a na q_a}
(q_\a-q)a(p_\a-p)=a(p_\a-p)
\end{equation}
and as a consequence (\ref{prelaz sa p_a na q_a v1}). Indeed, using $qa(1-p)=0$, we find $qa=qap$ and hence
$qap_\a=qapp_\a=qap$. Thus $(q_\a-q)a(p_\a-p)=q_\a ap_\a-q_\a ap-qap_\a+qap=q_\a ap_\a-q_\a a p=q_\a
a(p_\a-p)$. Thus, (\ref{prelaz sa p_a na q_a}) becomes $q_\a a(p_\a-p)=a(p_\a-p)$ which is equivalent to
(\ref{prelaz sa p_a na q_a v1}).

Since $p_{\alpha}-p\leq 1-p$ we have $qa(p_{\alpha}-p)=0$, and from (\ref{prelaz sa p_a na q_a v1}) we have
$$0=(1-q_\a)a(p_\a-p)=(1-q_\a)a(p_\a-p)+qa(p_{\alpha}-p)=(1-(q_\a-q))a(p_\a-p).$$
By Lemma \ref{razlaganje} $a$ is invertible up to $(p_\a,q_\a)$. Also, $q'_\a\sim p_\a-p\in\mathcal
F$ and therefore $q_\a=q'_\a+q\in\mathcal F$, as well. The first claim is proved.

Let us prove that $q_\a$ is a left approximate unit for $\mathcal F$. Let $a'=(1-q)a(1-p)$, and let
$f\in\mathcal F$. Since $qa(1-p)=0$, there holds $a'=a-(1-q)ap-qap$. Therefore
$(1-q_\a)a'=(1-q_\a)a-(1-q_a)ap$, since $q_\a\ge q$. Now, by (\ref{prelaz sa p_a na q_a v1}), we have
$(1-q_\a)a'f=(1-q_\a)a(1-p)f=(1-q_\a)a(1-p_\a)f\to0$ in norm topology for any $f\in\mathcal F$, since
$(1-q_\a)$ is norm bounded and $p_\a$ is an approximate unit. Thus $q_\a$ is a left approximate unit for
$a'\mathcal F$. However, any $f\in\mathcal F$ can be expressed as $f=(1-q)f+qf=a'bf+qf$. Using $q\le q_\a$ we
get $(1-q_\a)q=0$ and therefore
$$(1-q_\a)f=(1-q_\a)a'bf\to0,\qquad\mbox{in norm}.$$

To finish the proof, note that the mapping $p_\a\mapsto q_\a$ preserves order. Indeed, if $p_\beta\ge p_\a$,
we have
$$(q_\a-q)a(p_\a-p)=a(p_\a-p),\qquad (q_\beta-q)a(p_\beta-p)=a(p_\beta-p),$$
where $q_\a-q$, $q_\beta-q$ have minimal property. Multiply the second equality by $(p_\a-p)$ and using
$(p_\beta-p)(p_\a-p)=p_\a-p$ we find $(q_\beta-q)a(p_\a-p)=a(p_\a-p)$, which implies $q_\beta\ge q_\a$ by
minimal property.

Since $\mathcal F$ is a $*$-ideal, we have that $q_\a$ is a right approximate unit, as well.
\end{proof}

\subsection*{Index and its properties}

First, we establish that the difference $[p]-[q]$ will not change as long as a fixed $a$ is invertible up to
$(p,q)$. We need such a result to define the index, exactly to be the mentioned difference in the sequel
Definition.

\begin{proposition}\label{korektnost indeksa}
Let $a\in\mathcal A$ be invertible up to $(p,q)$, and also invertible up to $(p',q')$,
and let $p$, $q$, $p'$, $q'\in\mathcal F$. Then in $K(\mathcal F)$ we have
$$[p]-[q]=[p']-[q'].$$
\end{proposition}

\begin{proof} Due to the Proposition \ref{trijagularizacija}, we may assume that $qa(1-p)=q'a(1-p')=0$. Since
the same proposition follows that assumption will not change classes $[q]$ and $[q']$.

By Lemma \ref{utapanje u jedinicu} there is an approximate unit $p_\a$ of $\mathcal F$ containing $p$. For
all $\a$ we have
$$p_\a=p+r_\a.$$
By Lemma \ref{ekvivalentan nekoj jedinici}, choose $\a$ large enough, such that $p''\sim p'$,
$||p''-p'||<\delta$ and
$$p_\a=p''+r'_\a,$$
where $\delta<1/2$, $\delta<1/(4C||a||\,||b'||)$, $C$ is the absolute constant from (\ref{ocena_unitarnog})
and $b'$ is $(p',q')$-inverse of $a$.

The previous two displayed formulae yields
\begin{equation}\label{pp'Kgroup}
[p]+[r_\a]=[p']+[r'_\a]
\end{equation}

By Proposition \ref{desna apr jed}, there is another approximate unit $q_\a\ge q$ such that $q_\a-q\sim
p_\a-p$ and such that (\ref{prelaz sa p_a na q_a v1}) holds. Let $s_\a=q_\a-q$. Then
\begin{equation}\label{rsim s}
q_\a=q+s_\a,\qquad r_\a\sim s_\a.
\end{equation}

Enlarging $\a$, if necessary, there is $q''\le q_\a$ such that $q''\sim q'$ and $||q''-q'||<\delta$ and
consequently
$$q_\a=q''+s'_\a.$$

Thus we have, also
\begin{equation}\label{qq'Kgroup}
[q]+[s_\a]=[q']+[s'_\a].
\end{equation}

By Lemma \ref{mala pretumbacija} b), using $||p''-p'||$, $||q''-q'||<1/(4C||a||\,||b'||)$, we conclude that
$a$ is invertible up to $(p'',q'')$. Also, if $b''$ is $(p'',q'')$-inverse of $a$, then $||b''||\le2||b'||$.

Next, we want to estimate $||(1-q_\a)a(p_\a-p'')||$. Using (\ref{prelaz sa p_a na q_a v1}) we have
\begin{eqnarray*}
(1-q_\a)a(p_\a-p'')&=&(1-q_{\alpha})a
(p_{\alpha}-p+p-p'+p'-p'')\\
&=&(1-q_{\alpha})a
(p_{\alpha}-p)+(1-q_{\alpha})a(p-p')+(1-q_{\alpha})a(p'-p'')\\
&=&(1-q_\a)a(p-p')+(1-q_\a)a(p'-p'').
\end{eqnarray*}
Enlarging $\a$, once again, we can assume that $||(1-q_\a)a(p-p')||<||a||\delta$. Thus, we have
$||(1-q_\a)a(p_\a-p'')||<2||a||\delta<1/(2C||b'||)\le1/(C||b''||)$. Since $C>1$, we can apply Lemma
\ref{razlaganje} to obtain $q_\a-q''\sim p_\a-p''$, i.e.
\begin{equation}\label{r'sims'}
r'_\a\sim s'_\a.
\end{equation}

Subtracting (\ref{pp'Kgroup}) and (\ref{qq'Kgroup}) we obtain
$$[p]+[r_\a]-[q]-[s_\a]=[p']+[r'_\a]-[q']-[s'_\a]$$
which finishes the proof, because $[r_\a]-[s_\a]=0$ by (\ref{rsim s}), and $[r'_\a]-[s'_\a]=0$ by
(\ref{r'sims'}).
\end{proof}

\begin{remark} Note that in all previous proofs, we obtain Murray von Neumann equivalency, whereas, at this
moment, we include the cancellation law in $K$ group, which produces that it can be $[p]=[q]$ though $p$ and
$q$ might not be Murray - von Neumann equivalent. Namely, the Grothendick functor expand initial equivalence relation, if the
underlying semigroup does not satisfy the cancellation law. This fact is known in $K$ theory as "stable
equivalency".
\end{remark}

\begin{definition}\label{definicijaFredholm}

Let $\mathcal F$ be finite type elements. We say that $a\in\mathcal A$ is of Fredholm type (or abstract
Fredholm element) if there are $p$, $q\in\mathcal F$ such that $a$ is invertible up to $(p,q)$. The index of
the element $a$ (or abstract index) is the element of the group $K(\mathcal F)$ defined by
$$\ind(a)=([p],[q])\in K(\mathcal F),$$
or less formally
$$\ind(a)=[p]-[q].$$
Note that the index is well defined due to the Proposition \ref{korektnost indeksa}.
\end{definition}

Now, we proceed deriving the properties of the abstract index, defined in the previous Definition.

\begin{proposition} The set of Fredholm type elements is open in $\mathcal A$ and the index is a locally
constant function.
\end{proposition}

\begin{proof} It follows immediately from Lemma \ref{mala pretumbacija}.
\end{proof}

\begin{proposition}\label{1 plus f je u F} a) Let $a\in\mathcal A$ be of Fredholm type, and let $f\in\mathcal
F$. Then $a+f$ is also of Fredholm type, and $\ind(a+f)=\ind a$.

b) If $f\in\mathcal F$, then $1+f$ is of Fredholm type, and $\ind(1+f)=0$. Moreover, there is $p\in\mathcal
F$ such that $1+f$ is invertible up to $(p,p)$.
\end{proposition}

\begin{proof} a) Let $p$, $q\in\mathcal F$ be projections such that $a$ is invertible up to $(p,q)$. By
Proposition \ref{desna apr jed} there are approximate units $p_\a\ge p$, $q_\a\ge q$ such that $a$ is
invertible up to $(p_\a,q_\a)$. Choose $\a$ large enough such that $||f||\leq||b_{\alpha}||^{-1}$,
where $b_{\alpha}$ is  $(p_{\alpha},q_{\alpha})$-inverse of $a$. Then,
$$(1-q_{\alpha})(a+f)(1-p_{\alpha})=(1-q_{\alpha})a(1-p_{\alpha})+(1-q_{\alpha})f(1-p_{\alpha})$$
and $||(1-q_{\alpha})f(1-p_{\alpha})\leq||f||\leq||b_{\alpha}||^{-1}$.
By Lemma
\ref{mala pretumbacija}, $a+f$ is also invertible up to $(p_\a,q_\a)$ and hence $\ind(a+f)=[p_\a]-[q_\a]=\ind
a$;

b) In this special case, $a=1$ and we can choose $q_\a=p_\a$.
\end{proof}

\begin{proposition}\label{OurAtkinson} a) If $a$ is of Fredholm type, then $a$ is invertible modulo $\mathcal F$;

b) Conversely, if $a$ is invertible modulo $\mathcal F$, then $a$ is of Fredholm type.
\end{proposition}

\begin{proof} a) Assume that $a$ is invertible up to $(p,q)$. Let $b$ be $(p,q)$-inverse of $a$. Then
$ab=(1-q)ab+qab=(1-q)a(1-p)b+qab=1-q+qab\in 1+\mathcal F$, and also
$ba=ba(1-p)+bap=b(1-q)a(1-p)+bap=1-p+bap\in1+\mathcal F$.

b) Let $ab_1=1+f_1$, $b_2a=1+f_2$, $f_1$, $f_2\in\mathcal F$. By the previous Proposition, there is $p\in\mathcal F$ such that $ab_1$ is invertible up to $(p,p)$. It follows that $a^*$ is left invertible up to $p$ (its left inverse is $c^*(1-p)b_1^*$, where $c\in(1-p)\mathcal A(1-p)$ is $(p,p)$-inverse of $ab_1$). Therefore, by Lemma \ref{lema 2.4.}, there is a projection $1-r\in\mathcal A$ such that $(1-r)a^*(1-p)=a^*(1-p)$, or
equivalently,
\begin{equation}\label{Atkinson}
(1-p)a=(1-p)a(1-r).
\end{equation}
Furthermore, $a$ is invertible up to $(p,r)$. It remains to prove $r\in\mathcal F$.

Considering $(1-p)a$ instead of $a$, we find that
$$b_2(1-p)a=b_2a-b_2pa=1+f_2-b_2pa=1+f_2'\in1+\mathcal F.$$
Again, by the previous Proposition, there is a $q\in\mathcal F$ such that $b_2(1-p)a$ is invertible up to
$(q,q)$. By Proposition \ref{trijagularizacija}, there is a projection $q'\in\mathcal F$ such that
$b_2(1-p)a$ is invertible up to $(q',q)$ and such that $(1-q)b_2(1-p)aq'=0$. The last is equivalent to
$$(1-q)b_2(1-p)a(1-q')=(1-q)b_2(1-p)a.$$
If $t$ is $(q',q)$-inverse of $b_2(1-p)a$, then
$$(1-q)r=t(1-q)b_2(1-p)a(1-q')r=t(1-q)b_2(1-p)ar=0,$$
since from (\ref{Atkinson}) we easily find $(1-p)ar=0$. Thus, $r=qr\in\mathcal F$, since $\mathcal F$ is an
ideal.
\end{proof}

\begin{theorem}[index theorem]\label{teorema o indeksu} Let $\mathcal A$  be a
$C^*$-algebra, and let $\mathcal F\subseteq\mathcal A$ be an algebra of finite type elements. If $t_1$ and
$t_2$ are Fredholm type elements then $t_1t_2$ is of Fredholm type as well. Moreover there holds
$$\ind(t_1t_2)=\ind t_1+\ind t_2.$$

In other words, If we denote the set of all Fredholm type elements by $\mathrm{Fred}(\mathcal F)$, then
$\mathrm{Fred}(\mathcal F)$ is a semigroup (with unit) with respect to multiplication, and the mapping $\ind$
is a homomorphism from $(\mathrm{Fred}(\mathcal F),\cdot)$ to $(K(\mathcal F),+)$.
\end{theorem}

\begin{proof} Let $p_1$, $q_1$, $p_2$, $q_2\in\mathcal F$ be projections such that $t_1$ is invertible up to
$(p_1,q_1)$ and $t_2$ is invertible up to $(p_2,q_2)$. By Proposition \ref{desna apr jed}, there are
approximate units $p_\a\ge p_2$, $q_\a\ge q_2$ such that $t_2$ is also invertible up to $(p_\a,q_\a)$
and
\begin{equation}\label{indext2}
\ind t_2=[p_2]-[q_2]=[p_\a]-[q_\a].
\end{equation}
By Lemma \ref{ekvivalentan nekoj jedinici} there is an $\a$ and $p'\sim p_1$ such that $p'\le q_\a$ and
$||p_1-p'||<1/2$. By Lemma \ref{mala pretumbacija} b), $t_1$ is invertible up to $(p',q_1)$, and therefore
$\ind t_1=[p']-[q_1]$

Next, by Proposition \ref{trijagularizacija} there is a projection $q'\sim q_1$ such that $t_1$ is invertible
up to $(p',q')$ and
\begin{equation}\label{indexTHtriang}
q't_1(1-p')=0.
\end{equation}
It follows
\begin{equation}\label{indext1}
\ind t_1=[p']-[q'].
\end{equation}

Let $r=q_\a-p'$. We have $r\le 1-p'$, and by Lemma \ref{lema 2.4.}, there is a minimal projection $s$ such
that
\begin{equation}\label{indexTHdefS}
st_1r=t_1r.
\end{equation}
Also $r\sim s$ and $t_1$ is invertible up to $(1-r,1-s)$. By (\ref{indexTHtriang}) and Lemma \ref{treci deo
leme 2.4} we conclude $s\le 1-q'$. In particular, $1-q'-s$ is a projection and $1-q'-s\le 1-s$. So, from
(\ref{indexTHdefS}) we conclude that $(1-s)t_1r=0$ and hence $(1-q'-s)t_1r=(1-q'-s)(1-s)t_1r=0$. Therefore,
by Lemma \ref{razlaganje} we obtain that $t_1$ is invertible up to $(q_\a,q'+s)$.

Now, it is easy to derive that $t_1(1-q_\a)t_2$ is invertible up to $(p_\a,q'+s)$. Indeed, if $v_1$ is
$(q_\a,q'+s)$-inverse of $t_1$, and $v_2$ is $(p_\a,q_\a)$-inverse of $t_2$, a straightforward calculation
yields that $v_2v_1$ is $(p_\a,q'+s)$-inverse of $t_1(1-q_\a)t_2$. Thus $t_1(1-q_\a)t_2$ is of Fredholm type
and
\begin{multline*}\ind(t_1(1-q_\a)t_2)=[p_\a]-[q']-[s]=[p_\a]-[q_\a]+[q_\a]-[q']-[r]=\\
[p_\a]-[q_\a]+[p']-[q']=\ind t_1+\ind t_2.
\end{multline*}

Since $t_1t_2$ and $t_1(1-q_\a)t_2$ differ by $t_1q_\a t_2\in\mathcal F$, by Proposition \ref{1 plus f je u
F} a) $t_1t_2$ is also of Fredholm type and its index is the same as that of $t_1(1-q_\a)t_2$. The proof is
complete.
\end{proof}

\section{Applications}

In this section we apply results obtained in the previous section in three different ways. Namely, we shall
prove corollaries concerning the ordinary Fredholm operators, the Fredholm operators in the sense of Atiyah
and Singer over the $II_\infty$ factors and Fredholm operators on Hilbert $C^*$-modules over a $C^*$-algebra.
All of these corollaries have been known for a long time. Nevertheless, we shall derive them in order to show
that results obtained in the previous section are a generalization of all of them.

\subsection*{Classic Fredholm operators on a Hilbert space}

\begin{corollary} Let $\mathcal A$ be the full algebra of all bounded operators on some infinite dimensional
Hilbert space $H$, and let $\mathcal F$ be the ideal of all compact operators. Then the couple $(\mathcal
A,\mathcal F)$ satisfy the conditions $(i)-(iii)$ of Definition \ref{definicija konacnih}. Hence, ordinary
Fredholm operators are the special case of Fredholm type elements defined in this note.
\end{corollary}

\begin{proof} $(i)$ The ideal of all compact operators is a selfadjoint ideal in $B(H)$.

$(ii)$ As it is well known, the set of all finite rank projections is an approximate unit for
compact operators. Even more, if the Hilbert space $H$ is separable, there is a countable approximate unit
for $\mathcal F$. In more details if $P_n$ is the projection on the space generated by $\{e_1,\dots,e_n\}$,
where $\{e_j\}$ is an countable complete orthonormal system, then $\{P_n\:|\:n\in\N\}$ is a countable
approximate unit;

$(iii)$ Any compact projection $P\in\mathcal F$ is a finite rank projection. Then $H\ominus PH\cong H$. Let
$V:H\to H\ominus PH$ be an isomorphism. Then $VQV^*$ is a projection equivalent to $Q$ such that $VQV^*(H)\perp
P(H)$. Therefore, $VQV^*+P$ is projection.
\end{proof}

\subsection*{Fredholm operators on a von Neumann algebra}

Our second application is devoted to Fredholm operators in the sense of Breuer \cite{Breuer68,Breuer69} on a
properly infinite von Neumann algebra. First, we give the Breuer definition

\begin{definition}\label{Brojerova definicija}
Let $\mathcal A$ be a von Neumann algebra, let $\Proj(A)$ be the set of all projections belonging to
$\mathcal A$, and let $\Proj_0(\mathcal A)$ be the set of all finite projections in $\mathcal A$ (i.e.\ those
projections that are not Murray von Neumann equivalent to any its proper subprojection).

The operator $T\in\mathcal A$ is said to be $\mathcal A$-Fredholm if $(i)$ $P_{\ker T}\in\Proj_0(\mathcal
A)$, where $P_{\ker T}$ is the projection to the subspace $\ker T$ and $(ii)$ There is a projection
$E\in\Proj_0(\mathcal A)$ such that $\ran(I-E)\subseteq\ran T$. The second condition ensures that $P_{\ker
T^*}$ also belongs to $\Proj_0(\mathcal A)$.

The index of an $\mathcal A$-Fredholm operator $T$ is defined as
\begin{equation}\label{BreuerIndex}
\ind T=\dim(\ker T)-\dim(\ker T^*)\in I(\mathcal A).
\end{equation}

Here, $I(\mathcal A)$ is the so called {\em index group} of a von Neumann algebra $\mathcal A$ defined as the
Grothendieck group of the commutative monoid of all representations of the commutant $\mathcal A'$ generated
by representations of the form $\mathcal A'\ni S\mapsto ES=\pi_E(S)$ for some $E\in\Proj_0(\mathcal A)$
\cite[Section 2]{Breuer68}. For a subspace $L$, its dimension $\dim L$ is defined as the class
$[\pi_{P_L}]\in I(\mathcal A)$ of the representation $\pi_{P_L}$, where $P_L$ is the projection to $L$.
\end{definition}

Before proving that the Breuer's $\mathcal A$-Fredholm operators is a special case of abstract Fredholm
operators, we list some well known facts concerning von Neumann algebras.

\begin{lemma}\label{VN1} The set $\Proj(\mathcal A)=\{p\in \mathcal A\:|\:p\mbox{ is a projection}\}$ is a complete lattice. In particular,
if $p$, $q\in\Proj(\mathcal A)$, then $p\vee q$ (the least upper bound of the set $\{p,q\}$) belongs to
$\Proj(\mathcal A)$.
\end{lemma}

\begin{proof}The proof can be found in \cite[Proposition V.1.1]{Takesaki} or \cite[I.9.2.1(ii)]{CrnaGuja}.
\end{proof}

The operator $T\in\mathcal A$ is called finite if the projection to the closure of its range
$P_{\overline{\ran T}}\in\Proj_0(\mathcal A)$. The set of all finite operators is denoted by $\mathfrak m_0$,
and its norm closure is denoted by $\mathfrak m$.

\begin{lemma}\label{VN3} The set $\mathfrak m$ is a selfadjoint, two sided ideal of $\mathcal A$. Also, $\mathfrak m$
is generated (as a closed selfadjoint two sided ideal) by $\Proj_0(\mathcal A)$.
\end{lemma}

\begin{proof}This was proved in the Section 3 of \cite{Breuer68}.
\end{proof}

\begin{lemma}\label{VN4} If $p$,
$q\in\Proj_0(\mathcal A)$ then $p\vee q\in\Proj_0(\mathcal A)$ as well.
\end{lemma}

\begin{proof} The proof can be
found in \cite[Proposition V.1.6]{Takesaki} or \cite[III.1.1.3]{CrnaGuja}.
\end{proof}

\begin{lemma}\label{VN5} Let $T\in \mathcal A$ and let $T=U|T|$ be its polar decomposition. Then both $U$, $|T|\in \mathcal A$.
\end{lemma}

\begin{proof}The proof can be found in \cite[Proposition II.3.14]{Takesaki} or \cite[I.9.2.1(iii)]{CrnaGuja}.
\end{proof}

\begin{lemma}\label{VN6} Let $0\le T\in\mathcal A$. Then all spectral projections of $T$ also belongs to $\mathcal A$.
\end{lemma}

\begin{proof} The proof can be found in \cite[I.9.2.1(iv)]{CrnaGuja}.
\end{proof}

\begin{lemma}\label{BreuerAtkinson} Let $\mathcal A$ be a von Neumann algebra which is not of finite type.
Then the operator $T\in\mathcal A$ is $\mathcal A$-Fredholm if and only if $T$ is invertible modulo
$\mathfrak m$.
\end{lemma}

\begin{proof}This is \cite[Theorem 1]{Breuer69}.
\end{proof}

\begin{lemma} Let $\mathcal A$ be a von Neumann algebra which is not of finite type, and let $T$ be $\mathcal
A$-Fredholm operator. Then $|T|$ is $\mathcal A$-Fredholm as well.
\end{lemma}

\begin{proof} Let $T=V|T|$ be the polar decomposition of $T$. By Lemma \ref{VN5}, $|T|\in\mathcal A$, and by
Lemma \ref{BreuerAtkinson} there is $S\in\mathcal A$ such that $ST$, $TS\in 1+\mathfrak m$, i.e.\ $S$ is
$\mathfrak m$-inverse of $T$. Then, as it easy to see, $SV$ is a left $\mathfrak m$-inverse of $|T|$, whereas
$V^*S^*$ is a right inverse of $|T|$. Therefore, $|T|$ is invertible modulo $\mathfrak m$, and again by Lemma
\ref{BreuerAtkinson}, $|T|$ is $\mathcal A$-Fredholm.
\end{proof}

\begin{lemma}\label{spektralna mera Fr}
Let $T\ge0$ be an $\mathcal A$-Fredholm operator, and let $E_T$ be its spectral measure. Then, there is
$\delta>0$ such that $E_T([0,\delta))\in\Proj_0(\mathcal A)$.
\end{lemma}

\begin{proof} By Definition \ref{Brojerova definicija}, there is a closed subspace $L\subseteq\ran T$, which
complement is finite (i.e.\ $I-P_L\in\Proj_0(\mathcal A)$). Consider the restriction $T_1=T|_{\ker T^\bot}$,
and define $L_1=T_1^{-1}(L)$. Since both $L_1$ and $L$ are closed, by open mapping theorem $T_1$ realizes a
topological isomorphism between $L_1$ and $L$. Hence, there is $\delta>0$ such that for all $\xi\in L_1$
there holds $||T\xi||>\delta||\xi||$. However, for $\xi\in\ran E_T([0,\delta))$ we have
$||T\xi||\le\delta||\xi||$, i.e.\ $L_1\cap\ran E_T([0,\delta))=\{0\}$. Thus, it suffices to prove that
$L_1^\bot$ is finite.

Consider $L_2=(\ker T)^\bot\ominus L_1$. Obviously, $T(L_2)\cap L=\{0\}$, implying that $(I-P_L)T$ maps $L_2$
injectively to a subspace of $L^\bot$. By Lemma \ref{VN5} applied to $(I-P_L)TP_{L_2}$ there is a partial
isometry $W\in\mathcal A$ that connects $L_2$ and the closure $L_3$ of $(I-P_L)T(L_2)$ which is a subspace of
$L^\bot$. Since $L^\bot$ is finite, both $L_3\subseteq L^\bot$ and $L_2\sim L_3$ are also finite. Since by
Definition \ref{Brojerova definicija} $\ker T$ is finite as well, it follows that $L_1^\bot$ is finite. The
proof is complete, because $E_T([0,\delta))\le P_{L_1^\bot}$.
\end{proof}

\begin{corollary}\label{AS}Let $\mathcal A$ be a properly infinite von Neumann algebra acting on a Hilbert space
$H$, and let $\mathfrak m$ be defined as above. Then the couple $(\mathcal A,\mathfrak m)$ satisfies the
conditions $(i)-(iii)$ of Definition \ref{definicija konacnih}.

Moreover, abstract Fredholm elements are generalized Fredholm operators in the sense of Breuer.
\end{corollary}

\begin{proof} By Lemma \ref{VN3}, $\mathfrak m$ is an $*$-ideal which proves $(i)$ in Definition \ref{definicija
konacnih}.

Let us prove that the set $\Proj_0(\mathcal A)$ is an approximate unit. First, this is a directed set.
Indeed, let $p$, $q\in\Proj_0(\mathcal A)$. Since projection in a von Neumann algebra makes a complete
lattice (Lemma \ref{VN1}), the projection $p\vee q\in \mathcal A$. By Lemma \ref{VN4} $p\vee
q\in\Proj_0(\mathcal A)$ as well. Let $T\in\mathfrak m$, and let $T=U|T|$ be its polar decomposition. By
Lemma \ref{VN5} $U$, $|T|\in\mathcal A$, as well as any spectral projection of $|T|$ also belongs to
$\mathcal A$. Pick $\e>0$, and denote $P_\e=E_{|T|}(\e,+\infty)$, where $E_{|T|}$ stands for the spectral
measure of $|T|$. Then $P_\e\le P_{\overline{\ran T}}$ implying $P_\e\in\Proj_0(\mathcal A)$. Also,
$||\,|T|(I-P_\e)||\le\e$, since the function $\lambda\mapsto\lambda(1-\chi_{(\e,+\infty)})$ is bounded by
$\e$. For all $P\ge P_\e$ we have $I-P\le I-P_\e$ and hence, $||U|T|(I-P)||=||U|T|(I-P_\e)(I-P)||\le\e$. We
proved $(ii)$ in Definition \ref{definicija konacnih}.

Since $\mathcal A$ is properly infinite, by \cite[Lemma 8]{Breuer69}, for $P$, $Q\in\Proj_0(\mathcal A)$,
there holds $I-P\sim I$, i.e.\ there is a a partial isometry $V$ such that $V^*V=I$, and $VV^*=I-P$. Then
$VQ$ is a partial isometry, required in the property $(iii)$ od Definition \ref{definicija konacnih}.

Using Proposition \ref{OurAtkinson} and Lemma \ref{BreuerAtkinson} we see that abstract Fredholm operators
with respect to $(\mathcal A,\mathfrak m)$ coincide with $\mathcal A$-Fredholm operators in the sense of
Breuer. Let us prove that their indices are equal.

Let $T$ be Fredholm operator. By Lemma \ref{spektralna mera Fr}, there is a $\delta>0$ such that
$P=E_{|T|}[0,\delta)\in\Proj_0(\mathcal A)$. The operator $|T|$ is bounded below on the space
$E_{|T|}[\delta,\infty)H=(I-P)H$, hence $T=V|T|$ is also bounded below on $(I-P)H$. Let $I-Q$ be the
projection on the closure of $T(I-P)H$. Then, obviously, $T$ is invertible up to $(P,Q)$.

%The spectral projection $P$ commutes with $|T|$; hence $P(H)$ is invariant for $|T|$,
We have $P=P_{\ker T}\oplus E_{|T|}(0,\delta)=P_{\ker T}\oplus P_1$, and $N_1=P_1H$ is invariant for $|T|$.
Since $V$ acts as a partial isometry on $N$, $T$ maps $N_1$ some subspace $N_2\sim N_1$. Thus, $P(H)=\ker
T\oplus P_{N_1}$ and $Q=\ker T^*\oplus P_{N_2}$, and $P_{N_1}\sim P_{N_2}$ and the result follows.
\end{proof}

\begin{remark} The condition $(iii)$ in Definition \ref{definicija konacnih} is not substantial. If it is suppressed, we
have not a semigroup, but a conditional addition. However, we can also form a corresponding $K$-group,
starting with a free group with $\Proj_0(\mathcal A)$ as generators and then using appropriate
identifications. Obtained group will be isomorphic to $I(\mathcal A)$, see \cite[Section 2]{Breuer68}
\end{remark}

\subsection*{Fredholm operators on the standard Hilbert module}

The last corollary is devoted to Fredholm theory on the standard Hilbert $C^*$-module $l^2(\mathcal B)$ over
a unital $C^*$-algebra $\mathcal B$. We list some definitions and known facts. For references and further
details the reader is referred to \cite{TrMa}.

Let $\mathcal B$ be a unital $C^*$-algebra. We define the right Hilbert module $l^2(\mathcal B)$ as
$$l^2(\mathcal B)=\{(a_n)_{n=1}^{\infty}\:|\:a_n\in \mathcal B,\sum_n \skp{a_n}{a_n}\mbox{
converges in norm}\}$$
equipped with the right action of $\mathcal B$ $(a_n)b=(a_nb)$ and with the inner product
$\skp{\cdot}{\cdot}:l^2(\mathcal B)\times l^2(\mathcal B)\to\mathcal B$ given by
$$\skp{(a_n)}{(b_n)}=\sum_{n=1}^{\infty}a_n^*b_n.$$

Let $B(l^2(\mathcal B))$ denote the set of all bounded operators on $l^2(\mathcal B)$, and let
$B^a(l^2(\mathcal B))$ denote the algebra of all operators from $B(l^2(\mathcal B))$ that have the adjoint.

Let $C_0(l^2(\mathcal B))$ denote the algebra generated by all operators of the form
$$\theta_{x,y}(z)=x\skp yz.$$
Its closure we will denote by $C(l^2(\mathcal B))$. Although such operators might not map bounded into
relatively compact sets it is common to call them compact operators.

Next, we quote definitions and statements given by Mingo \cite{Mingo} and Mischenko and Fomenko
\cite{Mischenko}.

\begin{definition}\label{Mingo definition} $T\in B^a(l^2(\mathcal B))$ is Fredholm in the sense of Mingo, if it is invertible modulo compact
operators. \cite[\S1.1. Definition]{Mingo}.
\end{definition}

\begin{proposition}\label{Mingo characterization} $T$ is Fredholm in the sense of Mingo
if and only if there is a compact perturbation $S$ of $T$ (i.e.\ $T-S$ is compact) such that $\ker S$ and
$\ker S^*$ are finitely generated projective. The difference
\begin{equation}\label{Mindex}
[\ker S]-[\ker S^*]\in K_0(\mathcal B)
\end{equation}
does not depend on the choice of $S$ \cite[\S1.4.]{Mingo}.
\end{proposition}

\begin{definition}\label{Mingo def index}
The index of a Fredholm operator $T$ is defined to be the difference (\ref{Mindex}) where $S$ is any compact
perturbation of $T$ \cite[\S1.4.]{Mingo}.
\end{definition}

\begin{definition}\label{Mischenko definition} $T$ is Fredholm in the sense of Mischenko and Fomenko (MF sense in further)
if ($i$) $T$ is adjointable; ($ii$) there are decompositions of the domain $l^2(\mathcal B)=\mathcal
M_1\tilde\oplus\mathcal N_1$ and of the codomain $l^2(\mathcal B)=\mathcal M_2\tilde\oplus\mathcal N_2$,
where $\mathcal N_1$, $\mathcal N_2$ are finitely generated projective submodules and $T$ with respect to
such decompositions has a matrix form $\displaystyle T=\left(\begin{matrix}T_1&0\\0&T_2\end{matrix}\right)$,
with $T_1$ is an isomorphism. \cite[Definition 2.7.4]{TrMa}.

The index of $T$ is $\ind T=[\mathcal N_1]-[\mathcal N_2]$ \cite[Definition 2.7.8]{TrMa}.
\end{definition}

\begin{proposition}\label{Mischenko characterization} If $T$ is invertible modulo compacts then $T$ is Fredholm in MF sense. \cite[Theorem
2.7.14]{TrMa}
\end{proposition}

Note that the other implication trivially holds. Namely, if $T$ is Fredholm in MF sense, then $\displaystyle
\left(\begin{matrix}T_1^{-1}&0\\0&0\end{matrix}\right)$ is the inverse of $T$ modulo compacts.

We also need the following Lemma.

\begin{lemma}\label{I-Pboundedbelow} Let $l^2(\mathcal B)=\mathcal M\tilde\oplus\mathcal N$, where $\mathcal N$ is finitely
generated and projective and $\mathcal M$ is complemented. Next, let $P$ be the orthogonal projection to
$\mathcal N$. Then $I-P$ is bounded below on $\mathcal M$.
\end{lemma}

\begin{proof} Suppose this is not true, i.e.\ that there is a sequence $x_n\in\mathcal M$, $||x_n||=1$,
$(I-P)x_n\to0$. Decompose $(I-P)x_n$ as
\begin{equation}\label{Decompose(I-P)x_n}
(I-P)x_n=y'_n+y''_n,\qquad y'_n\in\mathcal M,y''_n\in\mathcal N.
\end{equation}
Then, $y'_n$, $y''_n\to0$, since $\mathcal M$ and $\mathcal N$ are complemented. However, from
(\ref{Decompose(I-P)x_n}) we obtain $\mathcal M\ni x_n-y'_n=Px_n+y''_n\in\mathcal N$. Thus both sides are
equal to $0$, since the sum $\mathcal M\tilde\oplus\mathcal N$ is direct. Hence $x_n=y'_n\to0$ which gives a
contradiction with $||x_n||=1$.
\end{proof}

\begin{corollary} Let $\mathcal A=l^2(\mathcal B)$ be the standard Hilbert module over some unital $C^*$-algebra $B$ and let
$\mathcal F=C(l^2(\mathcal B))$, the set of all compact operators on $l^2(\mathcal B)$. Then the couple
$(\mathcal A,\mathcal F)$ satisfies conditions in Definition \ref{definicija konacnih}.

Moreover, the set of abstract Fredholm operators, Fredholom operators in the sense of Mingo, and in the MF
sense coincide. Finally all three indices are equal.
\end{corollary}

\begin{proof}
($i$) It is well known that $\mathcal F$ is a closed ideal.

($ii$) Let $l^2(\mathcal B)$  possesses the standard basis $(e_i), i\in\mathbb{N}$, where
$e_i=(0,...,0,1,0,...)$ with the unit being the i-th entry.

Let $P_n\in\mathcal F$ be sequence of projections on $B^n$, $P_n(x_1,x_2,\dots)=(x_1,\dots,x_n,0,\dots)$ and
let $K\in C(l^2(\mathcal B))$. By \cite[Proposition  2.2.1]{TrMa} we have
$$||K-KP_n||\rightarrow0,\,\,n\rightarrow\infty.$$
Therefore, $P_n$ is an approximate unit.

($iii$) Let $P$, $Q\in\mathcal F$ be projections. By \cite[Theorem 15.4.2]{Wegge}, $P(\mathcal A)$ and
$Q(\mathcal A)$ are isomorphic (as modules) to some direct summand in $\mathcal B^n$ and  $\mathcal B^m$,
respectively, for some $m,n\in\mathbb{N}$. Hence, there is partial isometry
$$V: l^2(\mathcal B)\rightarrow l^2(\mathcal B),\quad V(x_1,x_2,...)=e_{m+1}x_1+e_{m+2}x_2+...+e_{m+n}x_n.$$
Then $P_1=VPV^*$ is orthogonal to $Q$. Consider the operator $U=VP$. We have
$$UU^*=VPV^*=P_1\qquad\text{and}\qquad U^*U=PV^*VP=P,$$
since $V^*V=P_{e_1,\dots,e_n}\ge P$. Also $(UU^*+Q)(UU^*+Q)=VPV^*VPV^*+Q^2=VPV^*+Q$ implying that $UU^*+Q$ is
a projection. Thus, property (iii) is proved.

Proposition \ref{OurAtkinson}, Proposition \ref{Mischenko characterization} together with a comment after it
and Definition \ref{Mingo definition} proves that all three kind of Fredholm operators (abstract, in the
sense of Mingo and in MF sense) coincide. Let us prove that their indices are equal.

Let $T$ be a Fredholm operator, and let $T$ be invertible up to $(P,Q)$. Then $T'=(I-Q)T(I-P)$ is a compact
perturbation of $T$ (its difference is $QT(I-P)+(I-Q)TP+QTP$), and $\ker T'=P(l^2(\mathcal B))$, $\ker
T'^*=Q(l^2(\mathcal B))$. Hence, abstract index is equal to Mingo's index.

Let $T$ be a Fredholm operator, and let $\displaystyle T=\left(\begin{matrix}T_1&0\\0&T_2\end{matrix}\right)$
with respect to decompositions $l^2(\mathcal B)=\mathcal M_1\tilde\oplus\mathcal N_1=\mathcal
M_2\tilde\oplus\mathcal N_2$ of the domain and the codomain, respectively. By \cite[Theorem 2.7.6]{TrMa},
$\mathcal M_1$, $\mathcal N_1$, $\mathcal M_2$, $\mathcal N_2$ can be chosen such that all of them are
complemented, $\mathcal N_1$ and $\mathcal N_2$ are finitely generated projective and also $\mathcal
M_1=\mathcal N_1^\bot$. (In truth, the last equality is not emphasized in the statement, but it follows from
the proof.) In the same proof we can found that $T|_{\mathcal M_1}$ is bounded below. Denote by $P$ and $Q$
the orthogonal projections to $\mathcal N_1$ and $\mathcal N_2$. By Lemma \ref{I-Pboundedbelow}, $I-Q$ is
bounded below on the space $\mathcal M_2=T(\mathcal M_1)=\ran T(I-P)$. This ensures that $(I-Q)T(I-P)$ is an
isomorphism from $\mathcal M_1=\mathcal N_1^\bot$ onto $\mathcal N_2^\bot$. Thus the abstract index is
$[P]-[Q]=[\mathcal N_1]-[\mathcal N_2]$, i.e.\ it is equal to MF index.
\end{proof}

\bibliographystyle{plain}
\bibliography{FredholmCstar}

\begin{thebibliography}{10}

\bibitem{Alvarez}
T.~Alvarez and M.~Onieva.
\newblock Generalized {F}redholm operators.
\newblock {\em Arch.~Math.}, 44:270--277, 1985.

\bibitem{Atiyah}
M.~F. Atiyah.
\newblock Elliptic operators, discrete groups and von {N}eumann algebras.
\newblock {\em Asterisque}, 32--33:43--72, 1976.

\bibitem{CrnaGuja}
B.~Blackadar.
\newblock {\em Operator Algebras - Theory of ${C}^*$-algebras and von Neumann
  algebras}, volume 122 of {\em Encyclopaedea of Mathematical Sciences}.
\newblock Springer, Berlin, Heidelberg, 2006.

\bibitem{Breuer68}
Manfred Breuer.
\newblock Fredholm theories in von neumann algebras ${I}$.
\newblock {\em Math. Annalen}, 178(3):243--254, 1968.

\bibitem{Breuer69}
Manfred Breuer.
\newblock Fredholm theories in von neumann algebras ${II}$.
\newblock {\em Math. Annalen}, 180(4):313--325, 1969.

\bibitem{Coburn}
L.~A. Coburn, R.~G. Douglas, D.~G. Schaeffer, and I.~M. Singer.
\newblock ${C}^*$-algebras of operators on a half-space, ii {I}ndex theory.
\newblock {\em Arch.~Publications mathematiques de l'I.H.\'{E}.S.}, 40:68--79,
  1971.

\bibitem{TrMa}
V.~M. Manuilov and E.~V. Troitsky.
\newblock {\em {H}ilbert $C^*$-modules}.
\newblock Translations of mathematical monographs Vol.\ 226. AMS, Providence,
  {R}hode {I}sland, 2005.

\bibitem{Mingo}
James~A. Mingo.
\newblock ${K}$-theory and multipliers of stable ${C}^*$-algebras.
\newblock {\em Trans. Amer. Math. Soc.}, 299(1):397--411, 1987.

\bibitem{Mischenko}
A.~S. Mishchenko and A.~T. Fomenko.
\newblock The index of elliptic operators over ${C}^*$-algebras.
\newblock {\em Math. USSR Izv.}, 15(1):87--112, 1980.

\bibitem{Marfi}
G.~J. Murphy.
\newblock {\em $C^*$-algebras and operator theory}.
\newblock Academic press, London, 1990.

\bibitem{Olsen}
Catherine~L. Olsen.
\newblock Index theory in von {N}eumann algebras.
\newblock {\em Mem. Amer. Math. Soc.}, 47(294):1--71, 1984.

\bibitem{Rudin}
W.~Rudin.
\newblock {\em Functional Analysis}.
\newblock McGraw-Hill Book Company, 1973.

\bibitem{RisNadjbook}
B.~Sz.-Nagy and F.~Riesz.
\newblock {\em Functional analysis}.
\newblock Frederick Ungar Publishing Co., New York, 1955.

\bibitem{Takesaki}
M.~Takesaki.
\newblock {\em Theory of Operator Algebras ${I}$}, volume 124 of {\em
  Encyclopaedea of Mathematical Sciences}.
\newblock Springer, Berlin, Heidelberg, etc., 2001.

\bibitem{Wegge}
N.~E. Wegge-Olsen.
\newblock {\em ${K}$-theory and ${C}^*$-algebras - A friendly approach}.
\newblock Oxford University Press, Oxford, New York, Tokyo, 1993.

\end{thebibliography}

\end{document}